\documentclass[12pt ,reqno]{article}

\usepackage[usenames,dvipsnames]{xcolor}
\usepackage{amssymb}
\usepackage{amsmath}
\usepackage{amsthm}
\usepackage{amsfonts}
\usepackage{amscd}
\usepackage{graphicx}

\usepackage[colorlinks=true,
linkcolor=webgreen,
filecolor=webbrown,
citecolor=webgreen]{hyperref}

\definecolor{webgreen}{rgb}{0,.5,0}
\definecolor{webbrown}{rgb}{.6,0,0}

\usepackage{color}
\usepackage{fullpage}
\usepackage{float}

\usepackage{graphics}
\usepackage{latexsym}
\usepackage{epsf}
\usepackage{breakurl}
\usepackage{cases}
\usepackage{colortbl}	
\usepackage{tikz}			
\usetikzlibrary{arrows.meta, calc, matrix, cd, positioning, shapes}
\usepackage{xfrac}
\usepackage{multicol}

\allowdisplaybreaks

\definecolor{shadecolor}{gray}{0.90}				
\def\boitegrise#1#2{\begin{centerline}{\fcolorbox{black}{shadecolor}{~
    \begin{minipage}[t]{#2}{\vphantom{~}#1\vphantom{$A_{\displaystyle{A_A}}$}}
            \end{minipage}~}}\end{centerline}\medskip}

\newcommand{\pmodd}[1]{\!\!\!\pmod{#1}}	
\newcommand{\PellSeq}{\left(P_n\right)_{n \geq 0}}
\newcommand{\QellSeq}{\left(Q_n\right)_{n \geq 0}}
\newcommand{\BalSeq}{\left(B_n\right)_{n \geq 0}}
\newcommand{\CobalSeq}{\left(b_n\right)_{n \geq 0}}
\newcommand{\LucBalSeq}{\left(C_n\right)_{n \geq 0}}
\newcommand{\LucCobalSeq}{\left(c_n\right)_{n \geq 0}}
\newcommand{\GenericSeq}{\left(S_n\right)_{n \geq 0}}
\newcommand{\FibSum}{\mathcal{F}(k)}

\newcommand{\seqnum}[1]{\href{https://oeis.org/#1}{\rm \underline{#1}}}

\DeclareFontFamily{U}{BOONDOX-calo}{\skewchar\font=45 }
\DeclareFontShape{U}{BOONDOX-calo}{m}{n}{
  <-> s*[1.05] BOONDOX-r-calo}{}
\DeclareFontShape{U}{BOONDOX-calo}{b}{n}{
  <-> s*[1.05] BOONDOX-b-calo}{}
\DeclareMathAlphabet{\mathcalb}{U}{BOONDOX-calo}{m}{n}
\SetMathAlphabet{\mathcalb}{bold}{U}{BOONDOX-calo}{b}{n}
\DeclareMathAlphabet{\mathbcalb}{U}{BOONDOX-calo}{b}{n}

\newcommand{\curl}[2]{\mathcalb{#1}^{#2}(k)}
\newcommand{\red}[1]{\textcolor{red}{#1}}
\newcommand{\blue}[1]{\textcolor{blue}{#1}}
\newcommand{\redbf}[1]{\textcolor{red}{\textbf{#1}}}

\definecolor{lightgray}{rgb}{0.90, 0.90, 0.90}

\begin{document}

% uncomment the below if paper accepted by Journal of Integer Sequences
%\begin{center}
%\epsfxsize=4in
%\leavevmode\epsffile{logo129.eps}
%\end{center}

\theoremstyle{plain}
\newtheorem{theorem}{Theorem}
\newtheorem{corollary}[theorem]{Corollary}
\newtheorem{lemma}[theorem]{Lemma}
\newtheorem{proposition}[theorem]{Proposition}
\newtheorem{result}[theorem]{Partial Result}

\theoremstyle{definition}
\newtheorem{definition}[theorem]{Definition}
\newtheorem{example}[theorem]{Example}
\newtheorem{conjecture}[theorem]{Conjecture}
\newtheorem{question}[theorem]{Question}
\newtheorem{convention}[theorem]{Convention}

\theoremstyle{remark}
\newtheorem{remark}[theorem]{Remark}

\begin{center}
\vskip 1cm{\LARGE\bf
Pell and Associated Pell Braid Sequences as GCDs of Sums of $k$ Consecutive Pell, Balancing, and Related Numbers
}
\vskip 1cm
aBa Mbirika\footnote{corresponding author} and Janee Schrader\\
Department of Mathematics\\
University of Wisconsin-Eau Claire\\
Eau Claire, WI 54702\\
USA \\
\href{mailto:mbirika@uwec.edu}{\tt mbirika@uwec.edu}\\
\href{mailto:schradjm5500@uwec.edu}{\tt schradjm5500@uwec.edu}\\
\ \\
J\"{u}rgen Spilker\\
Institute of Mathematics\\
University of Freiburg\\ 
79085 Freiburg im Breisgau \\
Germany\\
\href{mailto:juergen.spilker@t-online.de }{\tt juergen.spilker@t-online.de}
\end{center}

\vskip .2in

\begin{abstract}
We consider the greatest common divisor (GCD) of all sums of $k$ consecutive terms of a sequence $(S_n)_{n\geq 0}$ where the terms $S_n$ come from exactly one of following six well-known sequences' terms: Pell $P_n$, associated Pell $Q_n$, balancing $B_n$, Lucas-balancing $C_n$, cobalancing $b_n$, and Lucas-cobalancing $c_n$ numbers. For each of the six GCDs, we provide closed forms dependent on $k$. Moreover, each of these closed forms can be realized as braid sequences of Pell and associated Pell numbers in an intriguing manner. We end with partial results on GCDs of sums of squared terms and open questions.
\end{abstract}

\section{Introduction}

In 2021, Guyer and Mbirika gave closed forms for the greatest common divisor (GCD) of all sums of $k$ consecutive generalized Fibonacci numbers~\cite{Guyer_Mbirika2021}. Further, in 2022, Mbirika and Spilker generalized those results to the setting of the GCD of all sums of $k$ consecutive squares of generalized Fibonacci numbers~\cite{Mbirika_Spilker2022}. In this current paper, we extend the results of 2021 to the following six well-known sequences: Pell $\PellSeq$, associated Pell $\QellSeq$, balancing $\BalSeq$, Lucas-balancing $\LucBalSeq$, cobalancing $\CobalSeq$, and Lucas-cobalancing $\LucCobalSeq$. Moreover, these GCDs can be realized as braid sequences. A \textit{braid sequence} arises when we intertwine two sequences. For example in Figure~\ref{fig:Pell_Qell_braid}, we intertwine the sequence  $\left(P_n\right)_{n \geq 1}$ in the top row with the associated Pell sequence $\left(Q_n\right)_{n \geq 1}$ in the bottom row.
\begin{figure}[H]
    \centering
\begin{tikzpicture}
\matrix (m) [matrix of math nodes,
             nodes in empty cells,
             nodes={minimum height=4ex,text depth=0.5ex},
             column sep={1.75cm,between origins},
             row sep={1.5cm,between origins}]
{
P_1 & P_2 & P_3 & P_4 & P_5 & P_6 & P_7 & P_8 & \cdots\\
Q_1 & Q_2 & Q_3 & Q_4 & Q_5 & Q_6 & Q_7 & Q_8 & \cdots\\
};
\draw[blue, very thick,->]	(m-1-1) -- node[fill=white, circle] {} (m-2-2);
\draw[red, very thick, dotted, ->]	(m-2-1) -- (m-1-2);
\draw[red,very thick, dotted, ->]	(m-1-2) -- node[fill=white, circle] {} (m-2-3);
\draw[blue, very thick,->]	(m-2-2) --  (m-1-3);

\draw[blue, very thick,->]	(m-1-3) -- node[fill=white, circle] {} (m-2-4);
\draw[red, very thick, dotted, ->]	(m-2-3) -- (m-1-4);
\draw[red, very thick, dotted, ->]	(m-1-4) -- node[fill=white, circle] {} (m-2-5);
\draw[blue,very thick,->]	(m-2-4) -- (m-1-5);

\draw[blue, very thick,->]	(m-1-5) -- node[fill=white, circle] {} (m-2-6);
\draw[red, very thick, dotted, ->]	(m-2-5) -- (m-1-6);
\draw[red,very thick, dotted, ->]	(m-1-6) -- node[fill=white, circle] {} (m-2-7);
\draw[blue, very thick,->]	(m-2-6) -- (m-1-7);

\draw[blue, very thick,->]	(m-1-7) -- node[fill=white, circle] {} (m-2-8);
\draw[red, very thick, dotted, ->]	(m-2-7) -- (m-1-8);
\draw[red,very thick, dotted, ->]	(m-1-8) -- node[fill=white, circle] {} (m-2-9);
\draw[blue, very thick,->]	(m-2-8) -- (m-1-9);
\end{tikzpicture}
    \caption{The braiding of $\left(P_n\right)_{n \geq 1}$ and $\left(Q_n\right)_{n \geq 1}$.}
    \label{fig:Pell_Qell_braid}
\end{figure}
\noindent The red-dotted path is the sequence $\bigl(\gcd(B_n, b_n) \bigr)_{n \geq 1}$ and the blue-solid path is the sequence $\bigl(\gcd(B_n, b_n + 1) \bigr)_{n \geq 1}$. Both braid sequences easily follow from the identities
$$ B_n = P_n Q_n, \;\;\;
    b_n = \begin{cases}
  P_n Q_{n-1},  & \text{if $n$ is even}; \\
  P_{n-1} Q_n,  & \text{if $n$ is odd},
\end{cases} \;\;\;
    b_n + 1 = \begin{cases}
  P_{n-1} Q_n,  & \text{if $n$ is even}; \\
  P_n Q_{n-1},  & \text{if $n$ is odd},
\end{cases} $$
and the facts that $\gcd(P_{n-1}, P_n) = 1$ and $\gcd(Q_{n-1}, Q_n) = 1$.

\boitegrise{
\begin{convention}\label{conv:notation_for_curl_GCD_closed_forms}
When $\GenericSeq$ is any of the six sequences $\PellSeq$, $\QellSeq$, $\BalSeq$, $\LucBalSeq$, $\CobalSeq$, or $\LucCobalSeq$, we establish the notation $\curl{S}{m}$ to denote the GCD of all sums of $k$ consecutive $m^{\mathrm{th}}$ powers of sequence terms, in the respective six settings:
\begin{align*}
    \curl{P}{m} = \gcd\left\lbrace \left( \sum_{i=1}^k P_{n+i}^m \right)_{n \geq 0}\right\rbrace, &\hspace{.25in} & \curl{Q}{m} = \gcd\left\lbrace \left( \sum_{i=1}^k Q_{n+i}^m \right)_{n \geq 0}\right\rbrace,\\
    \curl{B}{m} = \gcd\left\lbrace \left( \sum_{i=1}^k B_{n+i}^m \right)_{n \geq 0}\right\rbrace, &\hspace{.25in} & \curl{C}{m} = \gcd\left\lbrace \left( \sum_{i=1}^k C_{n+i}^m \right)_{n \geq 0}\right\rbrace,\\
    \curl{b}{m} = \gcd\left\lbrace \left( \sum_{i=1}^k b_{n+i}^m \right)_{n \geq 0}\right\rbrace, &\hspace{.25in} & \curl{c}{m} = \gcd\left\lbrace \left( \sum_{i=1}^k c_{n+i}^m \right)_{n \geq 0}\right\rbrace.
\end{align*}
\begin{remark}
    When $m=1$, we omit the superscript and simply write $\curl{S}{}$.
\end{remark}
\end{convention}}{0.95\textwidth}

In this paper we give closed forms for the braid sequence of GCD-values $\bigl(\curl{S}{m}\bigr)_{k \geq 1}$ when $m=1$ for each of the six sequences and give partial results for when $m=2$. We chose these six sequences, in particular, since the GCD-values $\bigl(\curl{S}{}\bigr)_{k \geq 1}$ for each of the six sequences all involve Pell and associated Pell numbers in an intriguing manner. The breakdown of this paper is as follows. In Section~\ref{sec:definitions}, we provide definitions of the six sequences and some historical origins of the relatively newer sequences $\BalSeq$, $\LucBalSeq$, $\CobalSeq$, and $\LucCobalSeq$. In Section~\ref{sec:identities}, we give preliminary identities used to prove our main results, which are in Sections~\ref{sec:main_results} and \ref{sec:main_results_for_cobalancing_sequence}. Finally, in Section~\ref{sec:further_results_and_open_questions},  we address our progress towards the $m=2$ setting and provide some open questions.

\subsection{Motivation: a new proof of an old result}

At the $20^{\mathrm{th}}$ \textit{International Conference on Fibonacci Numbers and Their Applications} in Sarajevo in 2022, Mbirika presented his and collaborator Guyer's results on the GCD of all sums of $k$ consecutive generalized Fibonacci numbers~\cite{Guyer_Mbirika2021}. Conference participant Florian Luca communicated to Mbirika an observation that leads to a simple proof in the Fibonacci setting when $k$ is even. In this setting, the Guyer-Mbirika result was the following:
$$\FibSum = \begin{cases}
F_{k/2}, &\text{if $k \equiv 0 \pmodd 4$}; \\
L_{k/2}, &\text{if $k \equiv 2 \pmodd 4$}, 
\end{cases}$$
where $\FibSum$ denotes the GCD of all sums of $k$ consecutive Fibonacci numbers. For ease of notation, set $\sigma_F(k,n)$ to be the sum $\sum_{i=1}^k F_{n+i}$. Luca's observations was the following:
\begin{align}
\sigma_F(k,n) = F_{n+k+2} - F_{n+2} =  \begin{cases}
    F_{k/2} L_{\left(k/2+2\right)+n}, &\text{if $k \equiv 0 \pmodd{4}$};\\
    L_{k/2} F_{\left(k/2+2\right)+n}, &\text{if $k \equiv 2 \pmodd{4}$}.
     \end{cases}\label{eq:Luca_observation}
\end{align}
From this identity, the Guyer-Mbirika result is easily proven. For example for $k=20$, using Identity~\eqref{eq:Luca_observation} and the fact that consecutive Lucas numbers are relatively prime, we have
\begin{align*}
    \mathcal{F}(20) = \gcd(\sigma_F(20,0), \sigma_F(20,1), \sigma_F(20,2), \ldots) &= \gcd(F_{10} L_{12}, F_{10} L_{13}, F_{10} L_{14}, \ldots)\\
        &= F_{10} \cdot \gcd(L_{12}, L_{13}, L_{14}, \ldots)\\
        &= F_{10},
\end{align*}
as expected. For Identity~\eqref{eq:Luca_observation}, Luca noted that the first equality is easily shown if we utilize the fact that $\sum_{i=1}^{k} F_{i} = F_{k+2} - 1$, and the second equality follows by the known result
\begin{align}
    F_a - F_b &= \begin{cases}
    F_{\frac{a-b}{2}} L_{\frac{a+b}{2}}, &\text{if $a-b \equiv 0 \pmodd{4}$};\\
    L_{\frac{a-b}{2}} F_{\frac{a+b}{2}}, &\text{if $a-b \equiv 2 \pmodd{4}$},
     \end{cases}\label{eq:Luca_known_result}
\end{align}
if we set $a:=n+k+2$ and $b:=n+2$. Identity~\eqref{eq:Luca_known_result} follows directly from a 1963 result proven by Ruggles~\cite[p.~77]{Ruggles1963}. In this current paper, we generalize  Identity~\eqref{eq:Luca_known_result} into the Pell and associated Pell settings in Lemmas~\ref{lem:Pell_s_plus_r} and \ref{lem:Qell_s_plus_r}. Moreover, using the latter two lemmas we generalize Identity~\eqref{eq:Luca_observation} to compute the sums $\sigma_S(k,n)$ for five (of our six) different sequences $(S_n)_{n \geq 0}$ in one of our main results given in Theorem~\ref{thm:sigma_S_cute}.

\section{Definitions of the six sequences and some remarks}\label{sec:definitions}

We first recall the recursive definitions of the six sequences used in this paper, and then we follow with their well-known Binet forms.

\begin{definition}\label{def:Pell_Qell_numbers}
The \textit{Pell sequence} $\PellSeq$ and the \textit{associated Pell sequence} $\QellSeq$ are defined by the recurrence relations $P_n = 2 P_{n-1} + P_{n-2}$ and $Q_n = 2 Q_{n-1} + Q_{n-2}$, respectively,
with initial conditions $P_0 = 0$, $P_1 = 1$, $Q_0 = 1$, and $Q_1 = 1$. In the OEIS, these are sequences \seqnum{A000129} and \seqnum{A001333}, respectively~\cite{Sloane-OEIS}.
\end{definition}

\begin{remark}
In the literature, there is unfortunately some discrepancy on the precise definition of the \textit{Pell-Lucas sequence}. Though many sources attribute the OEIS sequence \seqnum{A002203} as the ``companion Pell sequence'' (or equivalently, the Pell-Lucas sequence), we choose to follow Koshy~\cite{Koshy2014} and many others in the literature who define the Pell-Lucas sequence as we have done in Definition~\ref{def:Pell_Qell_numbers} and call $\QellSeq$ the ``associated Pell sequence''.
\end{remark}

\begin{remark}
The associated Pell (respectively, Pell) sequence is the sequence of numerators (respectively, denominators) of the rational convergents to $\sqrt{2}$; that is, $\lim\limits_{n \rightarrow \infty} \frac{Q_n}{P_n} = \sqrt{2}$.
\end{remark}

Before we give the recursive definition of the remaining four sequences, we first discuss how these four sequences were originally defined. In 1999, Behera and Panda~\cite{Behera-Panda1999} defined an integer $n \in \mathbb{N}$ to be a balancing number if it is a solution to the Diophantine equation
\begin{align}
1 + 2 + \cdots + (n-1) = (n+1) + (n+2) + \cdots + (n + r), \label{eq:balancing_Diophantine}
\end{align}
where $r$ is the balancer corresponding to $n$. The terms in the sequence of balancing numbers and their corresponding balancers are denoted $B_n$ and $R_n$, respectively. For example, $B_2 = 6$ and $R_2 = 2$ since $1 + 2 + \cdots + 5 = 7 + 8$. Later in 2005, Panda and Ray~\cite{Panda_Ray2005} slightly modified Equation~\eqref{eq:balancing_Diophantine} to the Diophantine equation
\begin{align}
1 + 2 + \cdots + n = (n+1) + (n+2) + \cdots + (n + r). \label{eq:cobalancing_Diophantine}
\end{align}
In this new setting, they called the value $n$ a cobalancing number and the corresponding $r$ a cobalancer. The terms in the sequence of cobalancing numbers and their corresponding cobalancers are denoted $b_n$ and $r_n$, respectively. It turns out that every balancer is also a cobalancing number in the following sense: $R_n = b_n$. Moreover, every cobalancer is also a balancing number in the following sense: $r_{n+1} = B_n$. Hence, in this paper we consider the sequences $\BalSeq$ and $\CobalSeq$ and neither $(R_n)_{n \geq 0}$ nor $(r_n)_{n \geq 0}$.

Behera and Panda also showed that $B_n$ is a balancing number if and only if $8 B_n^2 + 1$ is a perfect square. So consider the sequence, denoted $\LucBalSeq$, of positive roots of $\sqrt{8 B_n^2 + 1}$ for each $n \geq 0$. This sequence is called the Lucas-balancing sequence and is named so since the value $C_n$ is associated to $B_n$ in many manners similar to the relationship between $L_n$ and $F_n$~\cite{Panda2009}. Lastly, it is known that $b_n$ is a cobalancing number if and only if $8 b_n^2 + 8 b_n + 1$ is a perfect square. So consider the sequence, denoted $\LucCobalSeq$, of positive roots of $\sqrt{8 b_n^2 + 8 b_n + 1}$ for each $n \geq 1$ and set $c_0 := -1$. This sequence is called the Lucas-cobalancing sequence.

We now give the recursive definitions of the four sequences $\BalSeq$, $\LucBalSeq$, $\CobalSeq$, and $\LucCobalSeq$. Then in Table~\ref{table:Pell_Qell_numbers}, we give the first eleven terms of each of the six sequences.

\begin{definition}\label{def:balancing_and_cobalancing_numbers}
The \textit{balancing sequence} $\BalSeq$ and the \textit{Lucas-balancing sequence} $\LucBalSeq$ are defined by the recurrence relations $B_n = 6 B_{n-1} - B_{n-2}$ and $C_n = 6 C_{n-1} - C_{n-2}$, respectively,
with initial conditions $B_0 = 0$, $B_1 = 1$, $C_0 = 1$, and $C_1 = 3$. In the OEIS, these are sequences \seqnum{A001109} and \seqnum{A001541}, respectively~\cite{Sloane-OEIS}.
\end{definition}

\begin{definition}\label{def:Lucas_balancing_and_Lucas_cobalancing_numbers}
The \textit{cobalancing sequence} $\CobalSeq$ and the \textit{Lucas-cobalancing sequence} $\LucCobalSeq$ are defined by the recurrence relations $b_n = 6 b_{n-1} - b_{n-2} + 2$ and $c_n = 6 c_{n-1} - c_{n-2}$, respectively,
with initial conditions $b_0 = 0$, $b_1 = 0$, $c_0 = -1$, and $c_1 = 1$. In the OEIS, these are sequences \seqnum{A053141} and \seqnum{A002315}, respectively~\cite{Sloane-OEIS}.
\end{definition}

Finally, let $\gamma = 1 + \sqrt{2}$ and $\delta = 1 - \sqrt{2}$. Then we have the following well-known Binet forms for the sequence terms $P_n$, $Q_n$, $B_n$, $C_n$, $b_n$, and $c_n$, respectively:

\begin{multicols}{3}
\begin{itemize}
    \item[] $P_n = \frac{\gamma^n - \delta^n}{2 \sqrt{2}}$,
    \item[] $Q_n = \frac{\gamma^n + \delta^n}{2}$,
    \item[] $B_n = \frac{\gamma^{2n} - \delta^{2n}}{4 \sqrt{2}}$,
    \item[] $C_n = \frac{\gamma^{2n} + \delta^{2n}}{2}$,
    \item[] $b_n = \frac{\gamma^{2n-1} - \delta^{2n-1}}{4 \sqrt{2}} - \frac{1}{2}$,
    \item[] $c_n = \frac{\gamma^{2n-1} + \delta^{2n-1}}{2}$.
\end{itemize}
\end{multicols}

\begin{table}[H]
\renewcommand{\arraystretch}{1.4}
\centering
\begin{tabular}{|c||c|c|c|c|c|c|c|c|c|c|c|c|}
\hline
\blue{$n$} & \redbf{0} & \redbf{1} & \redbf{2} & \redbf{3} & \redbf{4} & \redbf{5} & \redbf{6} & \redbf{7} & \redbf{8} & \redbf{9} & \redbf{10}\\ \hline\hline
\rowcolor{lightgray}
\blue{$P_n$} & 0 & 1 & 2 & 5 & 12 & 29 & 70 & 169 & 408 & 985 & 2378\\ \hline
\blue{$Q_n$} & 1 & 1 & 3 & 7 & 17 & 41 & 99 & 239 & 577 & 1393 & 3363\\ \hline\hline
\rowcolor{lightgray}
\blue{$B_n$} & 0 & 1 & 6 & 35 & 204 & 1189 & 6930 & 40391 & 235416 & 1372105 & 7997214\\ \hline
\blue{$C_n$} & 1 & 3 & 17 & 99 & 577 & 3363 & 19601 & 114243 & 665857 & 3880899 & 22619537\\ \hline\hline
\rowcolor{lightgray}
\blue{$b_n$} & 0 &	0 & 2 & 14 & 84 & 492 & 2870 & 16730 & 97512 & 568344 & 3312554\\ \hline
\blue{$c_n$} & $-1$ & 1 & 7 & 41 & 239 & 1393 & 8119 & 47321 & 275807 & 1607521 & 9369319\\ \hline
\end{tabular}
\caption{The first 11 Pell $P_n$, associated Pell $Q_n$, balancing $B_n$, cobalancing $b_n$, Lucas-balancing $C_n$, and Lucas-cobalancing $c_n$ numbers.}
\label{table:Pell_Qell_numbers}
\end{table}

\section{Some old and new identities}\label{sec:identities}
In this section, we provide the preliminary identities used to prove our main results in Sections~\ref{sec:main_results} and \ref{sec:main_results_for_cobalancing_sequence}. Some of these identities are well known, but most are new.

\subsection{Sum identities, Cassini's identities, and GCD identities}\label{subsec:some_known_identities}

The following lemma follows from results in Koshy's book~\cite{Koshy2014}, the Binet formulas given in Section~\ref{sec:definitions}, or Catarino et al.~\cite{Catarino2015}.

\begin{lemma}\label{lem:sum_S_i}
For all $k \geq 1$, the following sum identities hold for the six sequences Pell $(P_n)_{n\geq 0}$, associated Pell $(Q_n)_{n\geq 0}$, balancing $(B_n)_{n\geq 0}$, Lucas-balancing $(C_n)_{n\geq 0}$,cobalancing $(b_n)_{n\geq 0}$, and Lucas-cobalancing $(c_n)_{n\geq 0}$:

\begin{minipage}[b][25ex][t]{0.4\textwidth}
\begin{align}
    \sum_{i=1}^k P_i &= \frac{1}{2} ( Q_{k+1} - 1 )\label{eq:sum_P_i}, \\
    \sum_{i=1}^k Q_i &= P_{k+1} - 1, \label{eq:sum_Q_i} \\
    \sum_{i=1}^k B_i &= \frac{1}{4} ( P_{2k+1} - 1 ),\label{eq:sum_B_i}
\end{align}
\end{minipage}
\hspace{.4in}
\begin{minipage}[b][25ex][t]{0.4\textwidth}
\begin{align}
    \sum_{i=1}^k C_i &= \frac{1}{2} ( Q_{2k+1} - 1 ),\label{eq:sum_C_i} \\
    \sum_{i=1}^{k} b_{i} &= \frac{1}{4}(b_{k+1} - b_k - 2k),\label{eq:sum_b_i} \\
    \sum_{i=1}^k c_i &= \frac{1}{2} ( Q_{2k} - 1 ).\label{eq:sum_c_i} 
\end{align}
\end{minipage}
\end{lemma}

\begin{proof}[Proof of Identity~\eqref{eq:sum_P_i}] This is well known (see~\cite[Identity~(10.1)]{Koshy2014}).
\end{proof}

\begin{proof}[Proof of Identity~\eqref{eq:sum_Q_i}] This is well known (see~\cite[Identity~(10.2)]{Koshy2014}).
\end{proof}

\begin{proof}[Proof of Identity~\eqref{eq:sum_B_i}]
By the Binet formulas, we have
$B_i = \frac{\gamma^{2i} - \delta^{2i}}{4 \sqrt{2}} = \frac{1}{2} \cdot \frac{\gamma^{2i} - \delta^{2i}}{2 \sqrt{2}} = \frac{1}{2} P_{2i}$.
It follows that $\sum_{i=1}^k B_i = \frac{1}{2} \sum_{i=1}^k P_{2i} = \frac{1}{2} \left( \frac{1}{2} P_{2k+1} - \frac{1}{2} \right) = \frac{1}{4} ( P_{2k+1} - 1 )$,
where the second equality holds by a well-known identity for the sum of the first $k$ even-indexed Pell numbers (see~\cite[Identity~(10.4)]{Koshy2014}).
\end{proof}

\begin{proof}[Proof of Identity~\eqref{eq:sum_C_i}]
By the Binet formulas, we have $C_i = \frac{\gamma^{2i} - \delta^{2i}}{2} = Q_{2i}$.
It follows that $\sum_{i=1}^k C_i = \sum_{i=1}^k Q_{2i} = \frac{1}{2} ( Q_{2k+1} - 1 )$,
where the second equality holds by a well-known identity for the sum of the first $k$ even-indexed associated Pell numbers (see~\cite[Identity~(10.5)]{Koshy2014}).
\end{proof}

\begin{proof}[Proof of Identity~\eqref{eq:sum_b_i}]
    This is well known (see~\cite[Proposition~3.6]{Catarino2015}).
\end{proof}

\begin{proof}[Proof of Identity~\eqref{eq:sum_c_i}]
By the Binet formulas, we have $c_i = \frac{\gamma^{2i-1} - \delta^{2i-1}}{2} = Q_{2i-1}$.
It follows that $\sum_{i=1}^k c_i = \sum_{i=1}^k Q_{2i-1} = \frac{1}{2} ( Q_{2k} - 1 )$,
where the second equality holds by a well-known identity for the sum of the first $k$ odd-indexed associated Pell numbers (see~\cite[Identity~(10.6)]{Koshy2014}).
\end{proof}

Cassini's identity for the Fibonacci numbers has an analogue in both the Pell and associated Pell settings. We use the following Cassini's identities in the next two lemmas to prove the closed forms of $\curl{P}{}$ and $\curl{Q}{}$ given in Theorems~\ref{thm:curl_P1} and \ref{thm:curl_Q1}, respectively.

\begin{lemma}[Cassini's identity for $\PellSeq$]\label{lem:Cassini_identity_Pell_version}
For all $k \geq 1$, we have $P_{k-1} P_{k+1} = P_k^2 + (-1)^k$.
\end{lemma}
\begin{proof}
See Horadam~\cite[Identity~(30)]{Horadam1971}.
\end{proof}

Koshy mentions Cassini's identity in the associated Pell setting~\cite[Identity~(35)]{Koshy2014}; however, he provides no proof. As we could not find a proof in the literature of this identity, we provide our own proof using the Binet formula for $Q_n$.

\begin{lemma}[Cassini's identity for $\QellSeq$]\label{lem:Cassini_identity_Qell_version}
For all $k \geq 1$, we have
$$Q_{k-1} Q_{k+1} = Q_k^2 + 2(-1)^{k-1}.$$
\end{lemma}
\begin{proof}
By the Binet formula for the associated Pell sequence, we have
\begin{align*}
    Q_{k+1} Q_{k-1} - Q^2_k &=  \left ( \frac{ \gamma^{k+1} + \delta^{k+1} }{ 2 } \right ) \left ( \frac{ \gamma^{k-1} + \delta^{k-1} }{ 2 } \right ) - \left ( \frac{ \gamma^k + \delta^k }{ 2 } \right )^2 \\
    &= \frac{\gamma^{2k} + \gamma^{k+1} \delta^{k-1} + \gamma^{k-1} \delta^{k+1} + \delta^{2k} }{4} - \frac{\gamma^{2k} + 2(\gamma \delta )^k + \delta^{2k} }{4} \\
    &= \frac{\gamma^{2}(\gamma \delta)^{k-1} + \delta^2(\gamma \delta)^{k-1} - 2(\gamma \delta)^k}{ 4 } \\
    &= \frac{(\gamma^2 + \delta^2)(-1)^{k-1} - 2(-1)^{k}}{ 4 }  & \text{(since $\gamma \delta = -1$)}\\
    &= \frac{6(-1)^{k-1} + 2(-1)^{k-1}}{ 4 } \\
    &= 2(-1)^{k-1},
\end{align*}
where the fifth equality holds since $3 = Q_2 = \frac{\gamma^2 + \delta^2}{2}$ implies $\gamma^2 + \delta^2 = 6$.
\end{proof}

\begin{lemma}\label{lem:gcd_Pells_Qells}
For all $n \geq 1$, we have the following five identities:

\begin{minipage}[b][13ex][t]{0.4\textwidth}
\begin{align}
    \gcd( P_n, P_{n+1} ) &= 1, \label{eq:gcd_Pell_one_apart}\\
    \gcd( Q_n, Q_{n+1} ) &= 1, \label{eq:gcd_Qell_one_apart}\\
    \gcd( P_{2n}, P_{2n+2}) &= 2, \label{eq:gcd_P_even} 
\end{align}
\end{minipage}
\hspace{.4in}
\begin{minipage}[b][13ex][t]{0.4\textwidth}
\begin{align}
    \gcd( P_{2n-1}, P_{2n+1}) &= 1, \label{eq:gcd_P_odd}  \\
    \gcd( Q_n, Q_{n+2}) &= 1. \label{eq:gcd_Q} 
\end{align}
\end{minipage}
\end{lemma}

\begin{proof}
Identity~\eqref{eq:gcd_Pell_one_apart} follows from Lemma~\ref{lem:Cassini_identity_Pell_version}, while Identity~\eqref{eq:gcd_Qell_one_apart} follows from Lemma~\ref{lem:Cassini_identity_Qell_version} and the fact that associated Pell numbers are always odd. Identities~\eqref{eq:gcd_P_even} and \eqref{eq:gcd_P_odd} hold by Fl\'orez et al.~\cite{Florez2018} in Proposition~2 part (2) if we set $x:=1$, and Identity~\eqref{eq:gcd_Q} holds by Fl\'orez et al.~\cite{Florez2018} in Proposition~2 part (1) if we set $x:=1$.
\end{proof}

\subsection{New identities used to prove our main results}\label{subsec:our_new_identities}

\begin{lemma}\label{lem:Pell_s_plus_r}
For all $s,r \geq 1$ where $s$ is even, the following identity holds:
$$P_{s+r} - P_r = \begin{cases}
2P_{s/2} Q_{s/2+r}, &\text{if $s \equiv 0 \pmodd 4$}; \\
2Q_{s/2} P_{s/2+r}, &\text{if $s \equiv 2 \pmodd 4$}. 
\end{cases}$$
\end{lemma}

\begin{proof}
Let $s,r \geq 1$ be given where $s$ is even. 

\begin{itemize}
\item[] \textbf{Case I.} Suppose $s \equiv 0 \pmod 4$. Then $\frac{s}{2}$ is even and hence $(\gamma\delta)^{s/2} = 1$. Observe that
\begin{align*}
    P_{s+r} - P_r &=  \frac{\gamma^{s+r} - \delta^{s+r}}{2\sqrt{2}} - \frac{\gamma^{r} - \delta^{r}}{2\sqrt{2}}\\
        &= \frac{1}{2\sqrt{2}} \left( \gamma^{s+r} - \delta^{s+r} - (\gamma\delta)^{s/2} (\gamma^{r} - \delta^{r}) \right) &\text{(since $(\gamma\delta)^{s/2} = 1$)}\\
        &= \frac{1}{2\sqrt{2}} \left( \gamma^{s/2} - \delta^{s/2}  \right) \left( \gamma^{s/2 + r} + \delta^{s/2 + r}  \right)\\
        &= 2 \cdot \frac{\gamma^{s/2} - \delta^{s/2}}{2 \sqrt{2}} \cdot \frac{\gamma^{s/2 + r} + \delta^{s/2 + r}}{2}\\
        &= 2 P_{s/2} Q_{s/2+r}.
\end{align*}

\item[] \textbf{Case II.}  Suppose $s \equiv 2 \pmod 4$. Then $\frac{s}{2}$ is odd and hence $(\gamma\delta)^{s/2} = -1$. Observe that
\begin{align*}
    P_{s+r} - P_r &=  \frac{\gamma^{s+r} - \delta^{s+r}}{2\sqrt{2}} - \frac{\gamma^{r} - \delta^{r}}{2\sqrt{2}}\\
        &= \frac{1}{2\sqrt{2}} \left( \gamma^{s+r} - \delta^{s+r} + (\gamma\delta)^{s/2} (\gamma^{r} - \delta^{r}) \right) &\text{(since $(\gamma\delta)^{s/2} = -1$)}\\
        &= \frac{1}{2\sqrt{2}} \left( \gamma^{s/2} + \delta^{s/2}  \right) \left( \gamma^{s/2 + r} - \delta^{s/2 + r}  \right)\\
        &= 2 \cdot \frac{\gamma^{s/2} + \delta^{s/2}}{2} \cdot \frac{\gamma^{s/2 + r} - \delta^{s/2 + r}}{2\sqrt{2}}\\
        &= 2 Q_{s/2} P_{s/2+r}.
\end{align*}\qedhere
\end{itemize}
\end{proof}

\begin{lemma}\label{lem:Qell_s_plus_r}
For all $s,r \geq 1$ where $s$ is even, the following identity holds:
$$Q_{s+r} - Q_r = \begin{cases}
4P_{s/2} P_{s/2+r}, &\text{if $s \equiv 0 \pmodd 4$}; \\
2Q_{s/2} Q_{s/2+r}, &\text{if $s \equiv 2 \pmodd 4$}.
\end{cases}$$
\end{lemma}

\begin{proof}
Let $s,r \geq 1$ be given where $s$ is even. 

\begin{itemize}
\item[] \textbf{Case I.} Suppose $s \equiv 0 \pmod 4$. Then $\frac{s}{2}$ is even and hence $(\gamma\delta)^{s/2} = 1$. Observe that
\begin{align*}
    Q_{s+r} - Q_r &=  \frac{\gamma^{s+r} + \delta^{s+r}}{2} - \frac{\gamma^{r} + \delta^{r}}{2}\\
        &= \frac{1}{2} \left( \gamma^{s+r} + \delta^{s+r} - (\gamma\delta)^{s/2} (\gamma^{r} + \delta^{r}) \right) &\text{(since $(\gamma\delta)^{s/2} = 1$)}\\
        &= \frac{1}{2} \left( \gamma^{s/2} - \delta^{s/2}  \right) \left( \gamma^{s/2 + r} - \delta^{s/2 + r}  \right)\\
        &= 4 \cdot \frac{\gamma^{s/2} - \delta^{s/2}}{2 \sqrt{2}} \cdot \frac{\gamma^{s/2 + r} - \delta^{s/2 + r}}{2 \sqrt{2}}\\
        &= 4P_{s/2} P_{s/2+r}.
\end{align*}

\item[] \textbf{Case II.} Suppose $s \equiv 2 \pmod 4$. Then $\frac{s}{2}$ is odd and hence $(\gamma\delta)^{s/2} = -1$. Observe that
\begin{align*}
    Q_{s+r} - Q_r &=  \frac{\gamma^{s+r} + \delta^{s+r}}{2} - \frac{\gamma^{r} + \delta^{r}}{2}\\
        &= \frac{1}{2} \left( \gamma^{s+r} + \delta^{s+r} + (\gamma\delta)^{s/2} (\gamma^{r} + \delta^{r}) \right) &\text{(since $(\gamma\delta)^{s/2} = -1$)}\\
        &= \frac{1}{2} \left( \gamma^{s/2} + \delta^{s/2}  \right) \left( \gamma^{s/2 + r} + \delta^{s/2 + r}  \right)\\
        &= 2 \cdot \frac{\gamma^{s/2} + \delta^{s/2}}{2} \cdot \frac{\gamma^{s/2 + r} + \delta^{s/2 + r}}{2}\\
        &= 2 Q_{s/2} Q_{s/2+r}.
\end{align*}\qedhere
\end{itemize}
\end{proof}

Using the latter Lemmas~\ref{lem:Pell_s_plus_r} and \ref{lem:Qell_s_plus_r}, we are now ready to prove our main sum identities in the following theorem, which we use to prove our main results in Section~\ref{sec:main_results}.

\begin{theorem}\label{thm:sigma_S_cute}
For all $k\geq 1$, set $\sigma_S (k,n) := \sum_{i=1}^{k} S_{n+i}$ where $(S_n)_{n\geq 0}$ is any sequence. Then the following identities hold for the five sequences Pell $(P_n)_{n\geq 0}$, associated Pell $(Q_n)_{n\geq 0}$, balancing $(B_n)_{n\geq 0}$, Lucas-balancing $(C_n)_{n\geq 0}$, and Lucas-cobalancing $(c_n)_{n\geq 0}$:
\begin{align}
\sigma_P (k,n) = \frac{1}{2} (Q_{n+k+1} - Q_{n+1}) &= \begin{cases}
    2 P_{k/2} P_{k/2+n+1}, &\text{if $k \equiv 0 \pmodd{4}$}; \\ 
    Q_{k/2} Q_{k/2+n+1}, &\text{if $k \equiv 2 \pmodd{4}$}.
\end{cases} \label{eq:sigma_P} \\
\nonumber\\
\sigma_Q (k,n) = P_{n+k+1} - P_{n+1} &= \begin{cases}
    2 P_{k/2} Q_{k/2+n+1}, &\text{if $k \equiv 0 \pmodd{4}$}; \\
    2 Q_{k/2} P_{k/2+n+1}, &\text{if $k \equiv 2 \pmodd{4}$}.
\end{cases} \label{eq:sigma_Q} \\
\nonumber\\
\sigma_B (k,n) = \frac{1}{4} \left (  P_{2k+2n+1} - P_{2n+1} \right ) &= \begin{cases}
     \frac{1}{2} P_{k} Q_{k+2n+1}, &\text{if $k$ is even}; \\
     \frac{1}{2} Q_{k} P_{k+2n+1}, &\text{if $k$ is odd}.
\end{cases} \label{eq:sigma_B}\\
\nonumber\\
\sigma_C (k,n) = \frac{1}{2} \left (  Q_{2k+2n+1} - Q_{2n+1} \right ) &= \begin{cases}
    2 P_{k} P_{k+2n+1}, &\text{if $k$ is even}; \\
    Q_{k} Q_{k+2n+1}, &\text{if $k$ is odd}.
     \end{cases} \label{eq:sigma_BIG_C}\\
\nonumber\\
\sigma_c (k,n) = \frac{1}{2} \left (  Q_{2k+2n} - Q_{2n} \right ) &= \begin{cases}
    2 P_{k} P_{k+2n}, &\text{if $k$ is even}; \\
    Q_{k} Q_{k+2n}, &\text{if $k$ is odd}.
     \end{cases} \label{eq:sigma_little_c}
\end{align}

\end{theorem}

\begin{proof}[Proof of Identity~\eqref{eq:sigma_P}]
Let $k\geq 2$ be given where $k$ is even. Observe that 
\begin{align*}
    \sigma_P (k,n) := \sum\limits_{i=1}^{k} P_{n+i} &= \sum\limits_{i=1}^{k+n} P_i - \sum\limits_{i=1}^{n} P_i  \\
    &= \frac{1}{2} (Q_{k+n+1} - Q_{n+1}), 
\end{align*}
where the last equality holds by Identity~\eqref{eq:sum_P_i} of Lemma~\ref{lem:sum_S_i}. By Lemma~\ref{lem:Qell_s_plus_r}, if we set $s := k$ and $r := n+1$, then we have 
$$\sigma_P (k,n) = \begin{cases}
\frac{1}{2} \left ( 4 P_{k/2} P_{k/2+n+1} \right ) = 2 P_{k/2} P_{k/2+n+1}, &\text{if $k \equiv 0 \pmodd{4}$}; \\
\frac{1}{2} \left ( 2 Q_{k/2} Q_{k/2+n+1} \right ) = Q_{k/2} Q_{k/2+n+1}, &\text{if $k \equiv 2 \pmodd{4}$}.
\end{cases}
$$
\end{proof}

\begin{proof}[Proof of Identity~\eqref{eq:sigma_Q}]
Let $k\geq 2$ be given where $k$ is even. Observe that 
\begin{align*}
    \sigma_Q (k,n) := \sum\limits_{i=1}^{k} Q_{n+i} &= \sum\limits_{i=1}^{k+n} Q_i - \sum\limits_{i=1}^{n} Q_i  \\
    &= P_{k+n+1} - P_{n+1}, 
\end{align*}
where the last equality holds by Identity~\eqref{eq:sum_Q_i} of Lemma~\ref{lem:sum_S_i}. By Lemma~\ref{lem:Pell_s_plus_r}, if we set $s := k$ and $r := n+1$, then we have 
$$\sigma_Q (k,n) = \begin{cases}
2 P_{k/2} Q_{k/2+n+1}, &\text{if $k \equiv 0 \pmodd{4}$}; \\
2 Q_{k/2} P_{k/2+n+1}, &\text{if $k \equiv 2 \pmodd{4}$}.
\end{cases}
$$
\end{proof}

\begin{proof}[Proof of Identity~\eqref{eq:sigma_B}]
Let $k\geq 1$ be given. Observe that 
\begin{align*}
    \sigma_B (k,n) := \sum\limits_{i=1}^{k} B_{n+i} &= \sum\limits_{i=1}^{k+n} B_i - \sum\limits_{i=1}^{n} B_i  \\
    &= \frac{1}{4} \left( P_{2k+2n+1} - P_{2n+1} \right),
\end{align*}
where the last equality holds by Identity~\eqref{eq:sum_B_i} of Lemma~\ref{lem:sum_S_i}. By Lemma~\ref{lem:Pell_s_plus_r}, if we set $s := 2k$ and $r := 2n+1$, then we have 
$$\sigma_B (k,n) = \begin{cases}
\frac{1}{4} \left ( 2 P_{2k/2} Q_{2k/2+2n+1} \right ) = \frac{1}{2} P_{k} Q_{k+2n+1}, &\text{if $k$ is even}; \\
\frac{1}{4} \left ( 2 Q_{2k/2} P_{2k/2+2n+1} \right ) = \frac{1}{2} Q_{k} P_{k+2n+1}, &\text{if $k$ is odd},
\end{cases}
$$
since $s \equiv 0 \pmod{4}$ if and only if $k$ is even, and $s \equiv 2 \pmod{4}$ if and only if $k$ is odd.
\end{proof}

\begin{proof}[Proof of Identity~\eqref{eq:sigma_BIG_C}]
Let $k\geq 1$ be given. Observe that 
\begin{align*}
    \sigma_C (k,n) := \sum\limits_{i=1}^{k} C_{n+i} &= \sum\limits_{i=1}^{k+n} C_i - \sum\limits_{i=1}^{n} C_i  \\
    &= \frac{1}{2} \left (  Q_{2k+2n+1} - Q_{2n+1} \right ),
\end{align*}
where the last equality holds by Identity~\eqref{eq:sum_C_i} of Lemma~\ref{lem:sum_S_i}. By Lemma~\ref{lem:Qell_s_plus_r}, if we set $s := 2k$ and $r := 2n + 1$, then we have
$$\sigma_C (k,n) = \begin{cases}
    \frac{1}{2} \left ( 4 P_{2k/2} P_{2k/2+2n+1} \right ) = 2 P_{k} P_{k+2n+1}, &\text{if $k$ is even}; \\
    \frac{1}{2} \left ( 2 Q_{2k/2} Q_{2k/2+2n+1} \right ) = Q_{k} Q_{k+2n+1}, &\text{if $k$ is odd},
     \end{cases}$$
since $s \equiv 0 \pmod{4}$ if and only if $k$ is even, and $s \equiv 2 \pmod{4}$ if and only if $k$ is odd.
\end{proof}

\begin{proof}[Proof of Identity~\eqref{eq:sigma_little_c}]
Let $k\geq 1$ be given. Observe that 
\begin{align*}
    \sigma_c (k,n) := \sum\limits_{i=1}^{k} c_{n+i} &= \sum\limits_{i=1}^{k+n} c_i - \sum\limits_{i=1}^{n} c_i  \\
    &= \frac{1}{2} \left (  Q_{2k+2n} - Q_{2n} \right ),
\end{align*}
where the last equality holds by Identity~\eqref{eq:sum_c_i} of Lemma~\ref{lem:sum_S_i}. By Lemma~\ref{lem:Qell_s_plus_r}, if we set $s := 2k $ and $r := 2n $, then we have
$$\sigma_c (k,n) = \begin{cases}
    \frac{1}{2} \left ( 4 P_{2k/2} P_{2k/2+2n} \right ) = 2 P_{k} P_{k+2n}, &\text{if $k$ is even}; \\
    \frac{1}{2} \left ( 2 Q_{2k/2} Q_{2k/2+2n} \right ) = Q_{k} Q_{k+2n}, &\text{if $k$ is odd},
     \end{cases}$$
since $s \equiv 0 \pmod{4}$ if and only if $k$ is even, and $s \equiv 2 \pmod{4}$ if and only if $k$ is odd.
\end{proof}

\section{Main results for \texorpdfstring{$\curl{P}{}$, $\curl{Q}{}$, $\curl{B}{}$, $\curl{C}{}$, and $\curl{c}{}$}{five sequences}}\label{sec:main_results}

These are the braids for braid sequences $\bigl(\curl{P}{}\bigr)_{k \geq 1}$ in solid blue and $\bigl(\curl{Q}{}\bigr)_{k \geq 1}$ in dotted red. In Theorems~\ref{thm:curl_P1} and \ref{thm:curl_Q1}, respectively, we give the proofs of these two braids.
\begin{center}
\begin{tikzpicture}
\matrix (m) [matrix of math nodes,
             nodes in empty cells,
             nodes={minimum height=4ex,text depth=0.5ex},
             column sep={1.75cm,between origins},
             row sep={1.5cm,between origins}]
{
1 & 2Q_1 & 1 & 2P_2 & 1 & 2Q_3 & 1 & 2P_4 & \cdots\\
1 & Q_1 & 1 & 2P_2 & 1 & Q_3 & 1 & 2P_4 & \cdots\\
};
\draw[blue, very thick,->]	(m-1-1) -- node[pos=.55, scale=.7, fill=white, circle] {} (m-2-2);
\draw[red, very thick, dotted, ->]	(m-2-1) -- (m-1-2);
\draw[red,very thick, dotted, ->]	(m-1-2) -- node[pos=.4, scale=.7, fill=white, circle] {} (m-2-3);
\draw[blue, very thick,->]	(m-2-2) --  (m-1-3);

\draw[blue, very thick,->]	(m-1-3) -- node[pos=.57, scale=.7, fill=white, circle] {} (m-2-4);
\draw[red, very thick, dotted, ->]	(m-2-3) -- (m-1-4);
\draw[red, very thick, dotted, ->]	(m-1-4) -- node[pos=.4, scale=.7, fill=white, circle] {} (m-2-5);
\draw[blue,very thick,->]	(m-2-4) -- (m-1-5);

\draw[blue, very thick,->]	(m-1-5) -- node[pos=.55, scale=.7, fill=white, circle] {} (m-2-6);
\draw[red, very thick, dotted, ->]	(m-2-5) -- (m-1-6);
\draw[red,very thick, dotted, ->]	(m-1-6) -- node[pos=.4, scale=.7, fill=white, circle] {} (m-2-7);
\draw[blue, very thick,->]	(m-2-6) -- (m-1-7);

\draw[blue, very thick,->]	(m-1-7) -- node[pos=.57, scale=.7, fill=white, circle] {} (m-2-8);
\draw[red, very thick, dotted, ->]	(m-2-7) -- (m-1-8);
\draw[red,very thick, dotted, ->]	(m-1-8) -- node[pos=.475, scale=.7, fill=white, circle] {} (m-2-9);
\draw[blue, very thick,->]	(m-2-8) -- (m-1-9);
\end{tikzpicture}
\end{center}

\begin{table}[H]
\renewcommand{\arraystretch}{1.4}
\centering
\begin{tabular}{|c||c|c|c|c|c|c|c|c|c|c|c|c|c|c|c|}
\hline
$k$  & \bf{1} & \bf{2} & \bf{3} & \bf{4} & \bf{5} & \bf{6} & \bf{7} & \bf{8} & \bf{9} & \bf{10} & \bf{11} & \bf{12} & \bf{13} & \bf{14}\\ \hline\hline
\rowcolor{lightgray}
\blue{$\curl{P}{}$}  & \blue{1} & \blue{1} & \blue{1} & \blue{4} & \blue{1} & \blue{7} & \blue{1} & \blue{24} & \blue{1} & \blue{41} & \blue{1} & \blue{140} & \blue{1} &  \blue{239}\\ \hline
\red{$\curl{Q}{}$}  & \red{1} & \red{2} & \red{1} & \red{4} & \red{1} & \red{14} & \red{1} & \red{24} & \red{1} & \red{82} & \red{1} & \red{140} & \red{1} & \red{478}\\ \hline
\end{tabular}
\caption{The first 14 terms of the sequences $\bigl(\curl{P}{}\bigr)_{k \geq 1}$ and $\bigl(\curl{Q}{}\bigr)_{k \geq 1}$.}
\label{table:CurlPell_CurlQell_numbers}
\end{table}

\begin{theorem}\label{thm:curl_P1}
For all $k \geq 1$,  the GCD of all sums of $k$ consecutive Pell numbers is
\begin{align*}
    \curl{P}{} = \begin{cases}
  2 P_{k/2},  & \text{if $k \equiv 0 \pmodd{4}$}; \\
  Q_{k/2},  & \text{if $k \equiv 2 \pmodd{4}$}; \\
  1,  & \text{if $k \equiv 1,3 \pmodd{4}$}.
\end{cases} 
\end{align*}
\end{theorem}

\begin{proof}\label{prop:curl_P1}
Let $k\geq 1$ be given. Recall $\sigma_P (k,n) = \sum\limits_{i=1}^{k} P_{n+i}$. So by definition of $\curl{P}{}$, we have 
\begin{equation}
    \curl{P}{} = \gcd\left( \sigma_P(k,0), \sigma_P(k,1), \sigma_P(k,2), \ldots  \right). \label{eq: curl P1}
\end{equation}

\begin{itemize}
\item[] \textbf{Case I.} Suppose $k \equiv 0 \pmod 4$. By Identity~\eqref{eq:sigma_P} of Theorem~\ref{thm:sigma_S_cute}, it follows that 
\begin{align*}
    \curl{P}{} &= \gcd( 2P_{k/2} P_{k/2+1}, 2P_{k/2} P_{k/2+2}, 2P_{k/2} P_{k/2+3}, \ldots ) \\
    &= 2P_{k/2} \cdot \gcd( P_{k/2+1}, P_{k/2+2}, P_{k/2+3}, \ldots ) \\
    &= 2P_{k/2},
\end{align*}
where the last equality holds by Identity~\eqref{eq:gcd_Pell_one_apart} of Lemma~\ref{lem:gcd_Pells_Qells}.

\item[] \textbf{Case II.} Suppose $k \equiv 2 \pmod 4$. By Identity~\eqref{eq:sigma_P} of Theorem~\ref{thm:sigma_S_cute}, it follows that 
\begin{align*}
    \curl{P}{} &= \gcd( Q_{k/2} Q_{k/2+1}, Q_{k/2} Q_{k/2+2}, Q_{k/2} Q_{k/2+3}, \ldots  ) \\
    &= Q_{k/2} \cdot \gcd( Q_{k/2+1}, Q_{k/2+2}, Q_{k/2+3}, \ldots ) \\
    &= Q_{k/2}, 
\end{align*}
where the last equality holds by Identity~\eqref{eq:gcd_Qell_one_apart} of Lemma~\ref{lem:gcd_Pells_Qells}.

\item[] \textbf{Case III.} Suppose $k$ is odd. Assume by way of contradiction that $p^j$ divides $\curl{P}{}$ for some prime $p$ with $j \geq 1$. By Equation~\eqref{eq: curl P1}, it follows that $p^j$ divides $\sigma_P(k,n)$ for all $n \geq 0$. Hence $p^j$ divides both $\sigma_P(k,1) - \sigma_P(k,0)$ and $\sigma_P(k,2) - \sigma_P(k,1)$. Observe that
\begin{align*}
    \sigma_P(k,1) - \sigma_P(k,0) &= (P_2 + P_3 + \cdots + P_{k+1}) - (P_1 + P_2 + \cdots + P_{k}) = P_{k+1} -  1, \text{and}\\
    \sigma_P(k,2) - \sigma_P(k,1) &=  (P_3 + P_4 + \cdots + P_{k+2}) - (P_2 + P_3 + \cdots + P_{k+1}) = P_{k+2} - 2.
\end{align*}
Hence $p^j$ divides both $P_{k+1} - 1$ and $P_{k+2} - 2$. Since $P_k + 2P_{k+1} = P_{k+2}$, we have
\begin{align*}
    P_k &= P_{k+2} - 2 P_{k+1}\\
        &= P_{k+2} - 2 - 2 P_{k+1} + 2\\
        &= (P_{k+2} - 2) - 2(P_{k+1} - 1),
\end{align*}
and therefore $p^j$ divides $P_k$. Thus $P_k P_{k+2} \equiv 0 \pmod{p^j}$. Moreover, since $p^j$ divides $P_{k+1} - 1$, we have $p^j$ divides $P_{k+1}^2 - 1$, and thus $P_{k+1}^2 + 1 \equiv 2 \pmod{p^j}$. By the Pell Cassini Identity, Lemma~\ref{lem:Cassini_identity_Pell_version}, we have $P_k P_{k+2} = P_{k+1}^2 + 1$ since $k$ is odd. It follows that $2 \equiv 0 \pmod{p^j}$ and thus $p^j$ divides $2$, forcing $p=2$ and $j = 1$. Since $p^j$ divides $P_{k+1} - 1$ and we know $p^j = 2$, this implies $P_{k+1} - 1$ is even. However, $k$ being odd implies $P_{k+1}$ is even and hence $P_{k+1} - 1$ is also odd, which yields a contradiction. Thus there exists no prime that divides $\curl{P}{}$ when $k$ is odd, and hence $\curl{P}{} = 1$ for all odd $k$.\qedhere
\end{itemize}
\end{proof}

\begin{theorem}\label{thm:curl_Q1}
For all $k \geq 1$,  the GCD of all sums of $k$ consecutive associated Pell numbers is
\begin{align*}
    \curl{Q}{} = \begin{cases}
  2 P_{k/2},  & \text{if $k \equiv 0 \pmodd{4}$}; \\
  2Q_{k/2},  & \text{if $k \equiv 2 \pmodd{4}$};\\
  1,  & \text{if $k \equiv 1,3 \pmodd{4}$}.
\end{cases}
\end{align*}
\end{theorem}

\begin{proof}\label{prop:curl_Q1}
Let $k\geq 1$ be given. Recall $\sigma_Q (k,n) = \sum\limits_{i=1}^{k} Q_{n+i}$. So by definition of $\curl{Q}{}$, we have 
\begin{equation}
    \curl{Q}{} = \gcd\left( \sigma_Q(k,0), \sigma_Q(k,1), \sigma_Q(k,2), \ldots  \right). \label{eq: curl Q1}
\end{equation}

\begin{itemize}
\item[] \textbf{Case I.} Suppose $k\equiv 0 \pmod 4$. By Identity~\eqref{eq:sigma_Q} of Theorem~\ref{thm:sigma_S_cute}, it follows that 
\begin{align*}
    \curl{Q}{} &= \gcd( 2P_{k/2} Q_{k/2+1}, 2P_{k/2} Q_{k/2+2}, 2P_{k/2} Q_{k/2+3}, \ldots ) \\
    &= 2P_{k/2} \cdot \gcd( Q_{k/2+1}, Q_{k/2+2}, Q_{k/2+3}, \ldots ) \\
    &= 2P_{k/2}, 
\end{align*}
where the last equality holds by Identity~\eqref{eq:gcd_Qell_one_apart} of Lemma~\ref{lem:gcd_Pells_Qells}.

\item[] \textbf{Case II.} Suppose $k\equiv 2 \pmod 4$. By Identity~\eqref{eq:sigma_Q} of Theorem~\ref{thm:sigma_S_cute}, it follows that 
\begin{align*}
    \curl{Q}{} &= \gcd( 2Q_{k/2} P_{k/2+1}, 2Q_{k/2} P_{k/2+2}, 2Q_{k/2} P_{k/2+3},\ldots  ) \\
    &= 2Q_{k/2} \cdot  \gcd( P_{k/2+1},  P_{k/2+2},  P_{k/2+3},\ldots ) \\
    &= 2Q_{k/2},
\end{align*}
where the last equality holds by Identity~\eqref{eq:gcd_Pell_one_apart} of Lemma~\ref{lem:gcd_Pells_Qells}.

\item[] \textbf{Case III.} Suppose $k$ is odd. Assume by way of contradiction that $p^j$ divides $\curl{Q}{}$ for some prime $p$ with $j \geq 1$. By Equation~\eqref{eq: curl Q1}, it follows that $p^j$ divides $\sigma_Q(k,n)$ for all $n \geq 0$. Hence $p^j$ divides both $\sigma_Q(k,1) - \sigma_Q(k,0)$ and $\sigma_Q(k,2) - \sigma_Q(k,1)$. Observe that
\begin{align*}
    \sigma_Q(k,1) - \sigma_Q(k,0) &= (Q_2 + Q_3 + \cdots + Q_{k+1}) - (Q_1 + Q_2 + \cdots + Q_{k}) = Q_{k+1} -  1, \text{and}\\
    \sigma_Q(k,2) - \sigma_Q(k,1) &=  (Q_3 + Q_4 + \cdots + Q_{k+2}) - (Q_2 + Q_3 + \cdots + Q_{k+1}) = Q_{k+2} - 3.
\end{align*}
Hence $p^j$ divides both $Q_{k+1} - 1$ and $Q_{k+2} - 3$. Since $Q_k + 2Q_{k+1} = Q_{k+2}$, we have 
\begin{align*}
    Q_k &= Q_{k+2} - 2Q_{k+1} \\
    &= Q_{k+2} - 3 - 2Q_{k+1} + 3 \\
    &= (Q_{k+2} - 3) - 2(Q_{k+1} - 1) + 1,
\end{align*}
and therefore $p^j$ divides $Q_k - 1$, and thus $Q_k \equiv 1 \pmod{p^j}$. Also, since $p^j$ divides $Q_{k+2} - 3$, we have $Q_{k+2} \equiv 3 \pmod{p^j}$. Therefore, $Q_k Q_{k+2} \equiv 3 \pmod{p^j}$. Moreover, since $p^j$ divides $Q_{k+1} - 1$, we have $p^j$ divides $Q_{k+1}^2 - 1$, and thus $Q_{k+1}^2 - 2 \equiv -1 \pmod{p^j}$. By the associated Pell Cassini Identity, Lemma~\ref{lem:Cassini_identity_Qell_version}, we have $Q_k Q_{k+2} = Q_{k+1}^2 - 2$ since $k$ is odd. It follows that $3 \equiv -1 \pmod{p^j}$, and hence $p^j$ divides 4, so $p=2$ is forced. Since $p^j$ divides $Q_{k+2} - 3$ and we know $p = 2$, this implies $Q_{k+2} - 3$ is even. However, for any $k$ we have $Q_{k+2}$ being odd and hence $Q_{k+2} - 3$ is also odd, which yields a contradiction. Thus there exists no prime that divides $\curl{Q}{}$ when $k$ is odd, and hence $\curl{Q}{} = 1$ for all odd $k$.\qedhere
\end{itemize}
\end{proof}

These are the braids for braid sequences $\bigl(\curl{B}{}\bigr)_{k \geq 1}$ in solid blue and $\bigl(\curl{C}{}\bigr)_{k \geq 1}$ in dotted red. In Theorems~\ref{thm:curl_B1} and \ref{thm:curl_C1}, respectively, we give the proofs of these two braids.
\begin{center}
\begin{tikzpicture}
\matrix (m) [matrix of math nodes,
             nodes in empty cells,
             nodes={minimum height=4ex,text depth=0.5ex},
             column sep={1.85cm,between origins},
             row sep={1.5cm,between origins}]
{
Q_1 & 2P_2 & Q_3 & 2P_4 & Q_5 & 2P_6 & Q_7 & 2P_8 & \cdots\\
Q_1 & \frac{1}{2}P_2 & Q_3 & \frac{1}{2}P_4 & Q_5 & \frac{1}{2}P_6 & Q_7 & \frac{1}{2}P_8 & \cdots\\
};
\draw[blue, very thick,->]	(m-1-1) -- node[pos=.55, scale=.7, fill=white, circle] {} (m-2-2);
\draw[red, very thick, dotted, ->]	(m-2-1) -- (m-1-2);
\draw[red,very thick, dotted, ->]	(m-1-2) -- node[scale=.7, fill=white, circle] {} (m-2-3);
\draw[blue, very thick,->]	(m-2-2) --  (m-1-3);

\draw[blue, very thick,->]	(m-1-3) -- node[pos=.55, scale=.7, fill=white, circle] {} (m-2-4);
\draw[red, very thick, dotted, ->]	(m-2-3) -- (m-1-4);
\draw[red, very thick, dotted, ->]	(m-1-4) -- node[scale=.7, fill=white, circle] {} (m-2-5);
\draw[blue,very thick,->]	(m-2-4) -- (m-1-5);

\draw[blue, very thick,->]	(m-1-5) -- node[pos=.55, scale=.7, fill=white, circle] {} (m-2-6);
\draw[red, very thick, dotted, ->]	(m-2-5) -- (m-1-6);
\draw[red,very thick, dotted, ->]	(m-1-6) -- node[scale=.7, fill=white, circle] {} (m-2-7);
\draw[blue, very thick,->]	(m-2-6) -- (m-1-7);

\draw[blue, very thick,->]	(m-1-7) -- node[pos=.55, scale=.7, fill=white, circle] {} (m-2-8);
\draw[red, very thick, dotted, ->]	(m-2-7) -- (m-1-8);
\draw[red,very thick, dotted, ->]	(m-1-8) -- node[scale=.7, fill=white, circle] {} (m-2-9);
\draw[blue, very thick,->]	(m-2-8) -- (m-1-9);
\end{tikzpicture}
\end{center}

\begin{table}[H]
\renewcommand{\arraystretch}{1.4}
\centering
\begin{tabular}{|c||c|c|c|c|c|c|c|c|c|c|c|c|c|c|c|}
\hline
$k$  & \bf{1} & \bf{2} & \bf{3} & \bf{4} & \bf{5} & \bf{6} & \bf{7} & \bf{8} & \bf{9} & \bf{10} & \bf{11} & \bf{12} & \bf{13} \\ \hline\hline
\rowcolor{lightgray}
\blue{$\curl{B}{}$}  & \blue{1} & \blue{1} & \blue{7} & \blue{6} & \blue{41} & \blue{35} & \blue{239} & \blue{204} & \blue{1393} & \blue{1189} & \blue{8119} & \blue{6930} & \blue{47321} \\ \hline
\red{$\curl{C}{}$}  & \red{1} & \red{4} & \red{7} & \red{24} & \red{41} & \red{140} & \red{239} & \red{816} & \red{1393} & \red{4756} & \red{8119} & \red{27720} & \red{47321} \\ \hline
\end{tabular}
\caption{The first 13 terms of the sequences $\bigl(\curl{B}{}\bigr)_{k \geq 1}$ and $\bigl(\curl{C}{}\bigr)_{k \geq 1}$.}
\label{table:CurlB_Curlb_numbers}
\end{table}

\begin{theorem}\label{thm:curl_B1}
For all $k \geq 1$,  the GCD of all sums of $k$ consecutive balancing numbers is
\begin{align*}
    \curl{B}{} = \begin{cases}
  \frac{1}{2}P_k,  & \text{if $k$ is even}; \\
   Q_k,  & \text{if $k$ is odd}.
\end{cases} 
\end{align*}
\end{theorem}

\begin{proof}\label{prop:curl_B1}
Let $k\geq 1$ be given. Recall $\sigma_B (k,n) = \sum\limits_{i=1}^{k} B_{n+i}$. So by definition of $\curl{B}{}$, we have 
\begin{equation}
    \curl{B}{} = \gcd\left( \sigma_B(k,0), \sigma_B(k,1), \sigma_B(k,2), \ldots  \right). \label{eq: curl B1}
\end{equation}

\begin{itemize}
\item[] \textbf{Case I.} Suppose $k$ is even. By Identity~\eqref{eq:sigma_B} of Theorem~\ref{thm:sigma_S_cute}, it follows that 
\begin{align*}
    \curl{B}{} &= \gcd\left( \frac{1}{2}P_k Q_{k+1}, \frac{1}{2}P_k Q_{k+3}, \frac{1}{2}P_k Q_{k+5},\ldots   \right) \\
    &= \frac{1}{2}P_k \cdot \gcd\left( Q_{k+1}, Q_{k+3},  Q_{k+5},\ldots   \right) \\
    &= \frac{1}{2}P_k,
\end{align*}
where the last equality holds by Identity~\eqref{eq:gcd_Q} of Theorem~\ref{lem:gcd_Pells_Qells}.

\item[] \textbf{Case II.} Suppose $k$ is odd. By Identity~\eqref{eq:sigma_B} of Theorem~\ref{thm:sigma_S_cute}, it follows that 
\begin{align*}
    \curl{B}{} &= \gcd\left( \frac{1}{2}Q_k P_{k+1}, \frac{1}{2}Q_k P_{k+3}, \frac{1}{2}Q_k P_{k+5},\ldots  \right) \\
    &= \frac{1}{2}Q_k \cdot \gcd\left(  P_{k+1},  P_{k+3},  P_{k+5},\ldots   \right) \\
    &= Q_k,
\end{align*}
where the last equality holds by Identity~\eqref{eq:gcd_P_even} of Theorem~\ref{lem:gcd_Pells_Qells}.\qedhere
\end{itemize}

\end{proof}

\begin{theorem}\label{thm:curl_C1}
For all $k \geq 1$,  the GCD of all sums of $k$ consecutive Lucas-balancing numbers is
\begin{align*}
    \curl{C}{} = \begin{cases}
  2P_k,  & \text{if $k$ is even}; \\
   Q_k,  & \text{if $k$ is odd}.
\end{cases} 
\end{align*}
\end{theorem}

\begin{proof}\label{prop:curl_C1}
Let $k\geq 1$ be given. Recall $\sigma_C (k,n) = \sum\limits_{i=1}^{k} C_{n+i}$. So by definition of $\curl{C}{}$, we have 
\begin{equation}
    \curl{C}{} = \gcd\left( \sigma_C(k,0), \sigma_C(k,1), \sigma_C(k,2), \ldots  \right). \label{eq: curl C1}
\end{equation}

\begin{itemize}
\item[] \textbf{Case I.} Suppose $k$ is even. By Identity~\eqref{eq:sigma_BIG_C} of Theorem~\ref{thm:sigma_S_cute}, it follows that 
\begin{align*}
    \curl{C}{} &= \gcd\left( 2P_k P_{k+1}, 2P_k P_{k+3}, 2P_k P_{k+5}, \ldots  \right) \\
    &= 2P_k \cdot \gcd( P_{k+1}, P_{k+3}, P_{k+5}, \ldots ) \\
    &= 2P_k,
\end{align*}
where the last equality holds by Identity~\eqref{eq:gcd_P_odd} of Theorem~\ref{lem:gcd_Pells_Qells}.

\item[] \textbf{Case II.} Suppose $k$ is odd. By Identity~\eqref{eq:sigma_BIG_C} of Theorem~\ref{thm:sigma_S_cute}, it follows that  
\begin{align*}
    \curl{C}{} &= \gcd\left( Q_k Q_{k+1}, Q_k Q_{k+3}, Q_k Q_{k+5}, \ldots  \right) \\
    &= Q_k \cdot \gcd( Q_{k+1}, Q_{k+3}, Q_{k+5}, \ldots ) \\
    &= Q_k,
\end{align*}
where the last equality holds by Identity~\eqref{eq:gcd_Q} of Theorem~\ref{lem:gcd_Pells_Qells}.\qedhere
\end{itemize}
\end{proof}

These are the braids for braid sequences $\bigl(\curl{b}{}\bigr)_{k \geq 1}$ in solid blue and $\bigl(\curl{c}{}\bigr)_{k \geq 1}$ in dotted red. In Theorem~\ref{thm:curl_c1}, we give the proof of the $\bigl(\curl{c}{}\bigr)_{k \geq 1}$ braid. However, we leave the proof of the $\bigl(\curl{b}{}\bigr)_{k \geq 1}$ braid to Section~\ref{sec:main_results_for_cobalancing_sequence}.
\begin{center}
\begin{tikzpicture}
\matrix (m) [matrix of math nodes,
             nodes in empty cells,
             nodes={minimum height=4ex,text depth=0.5ex},
             column sep={1.75cm,between origins},
             row sep={1.5cm,between origins}]
{
\gcd(P_1,1) & 4P_2 & \gcd(P_3,3) & 4P_4 & \gcd(P_5,5) & 4P_6 & \gcd(P_7,7) & 4P_8 & \cdots\\
Q_1 & 2\gcd(Q_2,2) & Q_3 & 2\gcd(Q_4,4) & Q_5 & 2\gcd(Q_6,6) & Q_7 & 2\gcd(Q_8,8) & \cdots\\
};
\draw[blue, very thick,->]	(m-1-1) -- node[scale=.7, fill=white, circle] {} (m-2-2);
\draw[red, very thick, dotted, ->]	(m-2-1) -- (m-1-2);
\draw[red,very thick, dotted, ->]	(m-1-2) -- node[pos=.45, scale=.7, fill=white, circle] {} (m-2-3);
\draw[blue, very thick,->]	(m-2-2) --  (m-1-3);

\draw[blue, very thick,->]	(m-1-3) -- node[scale=.7, fill=white, circle] {} (m-2-4);
\draw[red, very thick, dotted, ->]	(m-2-3) -- (m-1-4);
\draw[red, very thick, dotted, ->]	(m-1-4) -- node[pos=.45, scale=.7, fill=white, circle] {} (m-2-5);
\draw[blue,very thick,->]	(m-2-4) -- (m-1-5);

\draw[blue, very thick,->]	(m-1-5) -- node[scale=.7, fill=white, circle] {} (m-2-6);
\draw[red, very thick, dotted, ->]	(m-2-5) -- (m-1-6);
\draw[red,very thick, dotted, ->]	(m-1-6) -- node[pos=.45, scale=.7, fill=white, circle] {} (m-2-7);
\draw[blue, very thick,->]	(m-2-6) -- (m-1-7);

\draw[blue, very thick,->]	(m-1-7) -- node[scale=.7, fill=white, circle] {} (m-2-8);
\draw[red, very thick, dotted, ->]	(m-2-7) -- (m-1-8);
\draw[red, very thick, dotted, ->]	(m-1-8) -- node[scale=.7, fill=white, circle] {} (m-2-9);
\draw[blue, very thick,->]	(m-2-8) -- (m-1-9);
\end{tikzpicture}
\end{center}

\begin{table}[H]
\renewcommand{\arraystretch}{1.4}
\centering
\begin{tabular}{|c||c|c|c|c|c|c|c|c|c|c|c|c|c|c|c|}
\hline
$k$  & \bf{1} & \bf{2} & \bf{3} & \bf{4} & \bf{5} & \bf{6} & \bf{7} & \bf{8} & \bf{9} & \bf{10} & \bf{11} & \bf{12} & \bf{13} \\ \hline\hline
\rowcolor{lightgray}
\blue{$\curl{b}{}$}  & \blue{2} & \blue{2} & \blue{2} & \blue{4} & \blue{2} & \blue{2} & \blue{2} & \blue{8} & \blue{2} & \blue{2} & \blue{2} & \blue{12} & \blue{2} \\ \hline
\red{$\curl{c}{}$}  & \red{1} & \red{8} & \red{7} & \red{48} & \red{41} & \red{280} & \red{239} & \red{1632} & \red{1393} & \red{9512} & \red{8119} & \red{55440} & \red{47321} \\ \hline
\end{tabular}
\caption{The first 13 terms of the sequences $\bigl(\curl{b}{}\bigr)_{k \geq 1}$ and $\bigl(\curl{c}{}\bigr)_{k \geq 1}$.}
\label{table:CurlC_Curlc_numbers}
\end{table}

\begin{theorem}\label{thm:curl_c1}
For all $k \geq 1$, the GCD of all sums of $k$ consecutive Lucas-cobalancing numbers is
\begin{align*}
    \curl{c}{} = \begin{cases}
  4P_k,  & \text{if $k$ is even}; \\
   Q_k,  & \text{if $k$ is odd}.
\end{cases} 
\end{align*}
\end{theorem}

\begin{proof}\label{prop:curl_c1}
Let $k\geq 1$ be given. Recall $\sigma_c (k,n) = \sum\limits_{i=1}^{k} c_{n+i}$. So by definition of $\curl{c}{}$, we have 
\begin{equation}
    \curl{c}{} = \gcd\left( \sigma_c(k,0), \sigma_c(k,1), \sigma_c(k,2), \ldots  \right). \label{eq: curl c1}
\end{equation}

\begin{itemize}
\item[] \textbf{Case I.} Suppose $k$ is even. By Identity~\eqref{eq:sigma_little_c} of Theorem~\ref{thm:sigma_S_cute}, it follows that 
\begin{align*}
    \curl{c}{} &= \gcd\left( 2P_k P_k, 2P_k P_{k+2}, 2P_k P_{k+4}, \ldots  \right) \\
    &= 2P_k \cdot \gcd( P_k, P_{k+2}, P_{k+4}, \ldots ) \\
    &= 4P_k,
\end{align*}
where the last equality holds by Identity~\eqref{eq:gcd_P_even} of Theorem~\ref{lem:gcd_Pells_Qells}.

\item[] \textbf{Case II.} Suppose $k$ is odd. By Identity~\eqref{eq:sigma_little_c} of Theorem~\ref{thm:sigma_S_cute}, it follows that 
\begin{align*}
    \curl{c}{} &= \gcd\left( Q_k Q_k, Q_k Q_{k+2}, Q_k Q_{k+4}, \ldots  \right) \\
    &= Q_k \cdot \gcd(Q_k, Q_{k+2}, Q_{k+4}, \ldots) \\
    &= Q_k,
\end{align*}
where the last equality holds by Identity~\eqref{eq:gcd_Q} of Theorem~\ref{lem:gcd_Pells_Qells}.\qedhere
\end{itemize}
\end{proof}

\section{Main results for \texorpdfstring{$\curl{b}{}$}{cobalancing sequence GCD}}\label{sec:main_results_for_cobalancing_sequence}

For the five sequences Pell $(P_n)_{n\geq 0}$, associated Pell $(Q_n)_{n\geq 0}$, balancing $(B_n)_{n\geq 0}$, Lucas-balancing $(C_n)_{n\geq 0}$, and Lucas-cobalancing $(c_n)_{n\geq 0}$, the closed forms of the GCD of all sums of $k$ consecutive terms involved braids of Pell and associate Pell numbers. However, in the setting of the cobalancing numbers, something much different occurs, namely we have the following closed form for the GCD of all sums of $k$ consecutive cobalancing numbers:
\begin{align}
    \curl{b}{} = \begin{cases}
  \gcd(P_k, k),  & \text{if $k$ is even}; \\
   2\gcd(Q_k, k),  & \text{if $k$ is odd}.
\end{cases} \label{eq:cobalancing_main_result}
\end{align}
To prove this, we first derive an intermediary form of $\curl{b}{}$ in Theorem~\ref{thm:curl_b1_intermediary_result} of Subsection~\ref{subsec:cobalancing_intermediary_result}. Then we prove our main result, Identity~\eqref{eq:cobalancing_main_result}, in Theorem~\ref{thm:curl_b1_main_result} of Subsection~\ref{subsec:cobalancing_main_result}.

\subsection{An intermediary result for \texorpdfstring{$\curl{b}{}$}{cobalancing sequence GCD}}\label{subsec:cobalancing_intermediary_result}

In the proof of Theorem~\ref{thm:curl_b1_intermediary_result}, we use the following easily-derived GCD result (see~\cite[Lemma~3.1]{Mbirika_Spilker2022}).
\begin{lemma}\label{lem:Jurgs_crazy_gcd_stuff}
    Let $(a_i)_{i \geq 0}$ be a sequence of integers. Then the following identity holds: 
    $$\gcd(a_0, a_1, a_2, a_3 \ldots) = \gcd(a_0, a_1 - a_0, a_2 - a_1, a_3 - a_2, \ldots).$$
\end{lemma}

\begin{theorem}\label{thm:sigma_b_cute}
For all $k\geq 1$, set $\sigma_b (k,n) := \sum\limits_{i=1}^{k} b_{n+i}$. Then the following identity holds:
\begin{align}
\sigma_b (k,n) = \frac{1}{2} (B_{k+n} - B_n - k).\label{eq:sigma_b}
\end{align}
\end{theorem}

\begin{proof}
    Let $k\geq 1$ be given. Observe that 
    \begin{align*}
        \sigma_b (k,n) := \sum_{i=1}^{k} b_{n+i} &= \sum_{i=1}^{k+n} b_{i} - \sum_{i=1}^{n} b_{i} \\
        &= \frac{1}{4} \left ( b_{k+n+1} - b_{k+n} - 2(k+n) \right ) - \frac{1}{4} \left ( b_{n+1} - b_n - 2n \right )\\
        &= \frac{1}{4} \left ( 2B_{k+n} - 2k -2n \right ) - \frac{1}{4} \left ( 2B_n - 2n \right ) \\
        &= \frac{1}{4} \left ( 2B_{k+n} - 2B_n - 2k \right ) \\
        &= \frac{1}{2} \left ( B_{k+n} - B_n - k \right ),
    \end{align*}
    where the second equality holds by Identity~\eqref{eq:sum_b_i} of Lemma~\ref{lem:sum_S_i}, and the third equality holds from the identity $2B_r = b_{r+1} - b_r$ for all $r \geq 1$ by Panda and Ray~\cite[Corollary~4.2]{Panda_Ray2005}. 
\end{proof}

\begin{lemma}\label{lem:t_n_result}
For all $i,k \geq 1$, the following identity holds:
$$ \frac{1}{2}(P_{2k+2i} - P_{2i}) - \frac{1}{2}(P_{2k+2i-2} - P_{2i-2}) = \begin{cases}
2 P_{k} Q_{k+2i-1}, &\text{if $k$ is even}; \\
2 Q_{k} P_{k+2i-1}, &\text{if $k$ is odd}.
\end{cases}
$$
\end{lemma}

\begin{proof}
For ease of notation, set $t_i := \frac{1}{2}(P_{2k+2i} - P_{2i}) - \frac{1}{2}(P_{2k+2i-2} - P_{2i-2})$. Then we have
\begin{align*}
t_i &= \frac{1}{2}\bigl((P_{2k+2i} - P_{2k+2i-2}) - (P_{2i} - P_{2i-2}) \bigr)\\
    &= \frac{1}{2} (2 P_{2k+2i-1} - 2 P_{2i-1}) &\text{(by the Pell recurrence)}\\
    &=  P_{2k+2i-1} - P_{2i-1}\\
    &= \begin{cases}
2 P_k Q_{k+2i}, &\text{if $k$ is even}; \\
2 Q_k P_{k+2i}, &\text{if $k$ is odd},
\end{cases}
\end{align*}
where the last equality holds by Lemma~\ref{lem:Pell_s_plus_r}, if we set $s:=2k$ and $r:=2i-1$ in the third equality. Then observe that $s \equiv 0 \pmod{4}$ if and only if $k$ is even, and $s \equiv 2 \pmod{4}$ if and only if $k$ is odd.
\end{proof}

\begin{theorem}\label{thm:curl_b1_intermediary_result}
For all $k \geq 1$, the GCD of all sums of $k$ consecutive cobalancing numbers is
\begin{align*}
    \curl{b}{} = \begin{cases}
  \gcd\left (\frac{1}{2}( B_k - k), P_k \right ),  & \text{if $k$ is even}; \\
   \gcd\left (\frac{1}{2} (B_k - k), 2Q_k \right ),  & \text{if $k$ is odd}.
\end{cases} 
\end{align*}
\end{theorem}

\begin{proof}
    Let $k \geq 1$ be given. Recall $\sigma_b (k,n) = \sum\limits_{i=1}^k b_{n+i}$. So by definition of $\curl{b}{}$, we have 
    \begin{align*}
    \curl{b}{} &= \gcd\left( \sigma_b(k,0), \sigma_b(k,1), \sigma_b(k,2), \ldots  \right)\\
    &= \gcd\left( \frac{1}{2} (B_k - k), \frac{1}{2} (B_{k+1} - B_1 - k), \frac{1}{2} (B_{k+2} - B_2 - k), \ldots \right),
\end{align*}
    where the second equality holds by Theorem~\ref{thm:sigma_b_cute}. For ease of notation, set $r_i := B_{k+i} - B_i - k$ and $s_i := B_{k+i} - B_i - (B_{k+i-1} - B_{i-1})$ and $t_i := \frac{1}{2}(P_{2k+2i} - P_{2i}) - \frac{1}{2}(P_{2k+2i-2} - P_{2i-2})$. It is clear that $s_i = r_i - r_{i-1}$ for all $i \geq 1$. Moreover by the Binet formulas, $B_i = \frac{P_{2i}}{2}$ holds, so we have $s_i = t_i$ for all $i \geq 1$. Observe that 
    \begin{align*}
        \curl{b}{} &= \frac{1}{2} \gcd\left ( r_0, r_1, r_2, \ldots \right ) \\
        &= \frac{1}{2} \gcd \left( B_k - k, s_1, s_2, \ldots \right ) &\text{(by Lemma~\ref{lem:Jurgs_crazy_gcd_stuff})} \\
        &= \frac{1}{2} \gcd\left ( B_k - k, \gcd(s_1, s_2, \ldots )  \right ) \\
        &= \frac{1}{2} \gcd\left ( B_k - k, \gcd(t_1, t_2, \ldots )  \right ). 
    \end{align*}

\begin{itemize}
\item[] \textbf{Case I.} Suppose $k$ is even. Then $t_i = 2 P_k Q_{k+2i-1}$  for all $i \geq 1$ by Lemma~\ref{lem:t_n_result}. It follows~that 
\begin{align*}
    \curl{b}{} &= \frac{1}{2} \gcd\left ( B_k - k, \gcd(2P_k Q_{k+1}, 2P_k Q_{k+3}, 2P_k Q_{k+5}, \ldots \right ) \\
    &= \frac{1}{2} \gcd\left ( B_k - k, 2P_k \cdot \gcd( Q_{k+1}, Q_{k+3}, Q_{k+5}, \ldots) \right ) \\
    &= \frac{1}{2} \gcd\left ( B_k - k, 2P_k \right ) \\
    &= \gcd\left (\frac{1}{2}( B_k - k), P_k \right ),
\end{align*}
where the third equality holds by Identity~\eqref{eq:gcd_Q} of Theorem~\ref{lem:gcd_Pells_Qells}.

\item[] \textbf{Case II.} Suppose $k$ is odd. Then $t_i = 2Q_k P_{k+2i-1}$ for all $i \geq 1$ by Lemma~\ref{lem:t_n_result}. It follows that 
\begin{align*}
    \curl{b}{} &= \frac{1}{2} \gcd\left ( B_k - k, \gcd(2Q_k P_{k+1}, 2Q_k P_{k+3}, 2Q_k P_{k+5}, \ldots) \right ) \\
    &= \frac{1}{2} \gcd\left ( B_k - k, 2Q_k \cdot \gcd( P_{k+1}, P_{k+3}, P_{k+5}, \ldots) \right ) \\
    &= \frac{1}{2} \gcd\left ( B_k - k, 4Q_k \right ) \\
    &= \gcd\left (\frac{1}{2} (B_k - k), 2Q_k \right ),
\end{align*}
where the third equality holds by Identity~\eqref{eq:gcd_P_even} of Theorem~\ref{lem:gcd_Pells_Qells}.\qedhere
\end{itemize}
\end{proof}

\subsection{Our main result for \texorpdfstring{$\curl{b}{}$}{cobalancing sequence GCD} involving \texorpdfstring{$\gcd(P_k,k)$}{gcd(Pk,k} and \texorpdfstring{$\gcd(Q_k,k)$}{gcd(Qk,k}}\label{subsec:cobalancing_main_result}

To prove the results in this subsection, we use the $p$-adic valuation function and some of its well-known properties in Lemma~\ref{lem:Carl_is_Filipino} whose proofs we omit. In this subsection, we sometimes apply Lemma~\ref{lem:Carl_is_Filipino} without reference.

\begin{definition}
    For each $n \geq 1$ and $p$ a prime, the \textit{$p$-adic valuation} of $n$, denoted $\nu_p(n)$, is the smallest nonnegative integer $k$ such that $p^k$ divides $n$.
\end{definition}

\begin{lemma}\label{lem:Carl_is_Filipino}
  For all $a,b \in \mathbb{Z}$ and $p$ a prime, the following identities hold:
  \begin{align}
  \nu_p( \gcd(a, b) ) &= \min\left(\nu_p(a), \nu_p(b)\right), \label{eq:p_adic_prop_1}\\
  \nu_p(a \cdot b) &= \nu_p(a) + \nu_p(b). \label{eq:p_adic_prop_2}
  \end{align}
\end{lemma}

\begin{lemma}\label{lem:Jurgen_p_adic_lemma_1}
    For all $k\geq 1$, we have $\nu_2 (P_k) = \nu_2 (k)$.
\end{lemma}

\begin{proof}
This follows from a more general result by Sanna (see~\cite[Theorem~1.5]{Sanna2016}).
\end{proof}

\begin{lemma}\label{lem:Jurgen_p_adic_lemma_2}
    For all $k\geq 2$, we have $\nu_2 (P_k Q_k - k) \geq 2$.
\end{lemma}

\begin{proof}
We first claim that if $\ell \geq 4$ is even, then $\nu_2(P_\ell - \ell) \geq 3$ holds. Observe that
    \begin{align}
        P_\ell &= \sum_{i=1}^{\ell/2} \binom{\ell}{2i-1} 2^{i-1}  \equiv \ell + \binom{\ell}{3} 2 + \binom{\ell}{5} 2^2 \pmodd{8},\label{eq:Pell_power_of_2_identity}
    \end{align}
where the first equality holds by~\cite[Identity~(9.10)]{Koshy2014}. For $\ell \geq 4$, we have $\binom{\ell}{3} = \frac{\ell \cdot (\ell - 1) \cdot (\ell - 2)}{2 \cdot 3}$ and thus $\nu_2\bigl(\binom{\ell}{3}\bigr) \geq 2$, so $\nu_2\bigl(\binom{\ell}{3} 2 \bigr) $$= \nu_2\bigl(\binom{\ell}{3}\bigr) + \nu_2(2)  \geq 2 + 1 = 3$ holds. Moreover for $\ell \geq 6$, we have $\binom{\ell}{5} = \frac{\ell \cdot (\ell - 1) \cdot (\ell - 2) \cdot (\ell - 3) \cdot (\ell - 4)}{2^3 \cdot 15}$ and thus $\nu_2\bigl(\binom{\ell}{5}\bigr) \geq 1$, so $\nu_2\bigl(\binom{\ell}{5} 2^2 \bigr) = \nu_2\bigl(\binom{\ell}{5}\bigr) + \nu_2(2^2)  \geq 1 + 2 = 3$ holds. Hence 8 divides both $\binom{\ell}{3} 2$ and $\binom{\ell}{5} 2^2$, and so Identity~\eqref{eq:Pell_power_of_2_identity} implies $P_\ell - \ell$ is also divisible by 8, and thus $\nu_2(P_\ell - \ell) \geq 3$, as desired. Now let $k \geq 2$ be given. Observe that
$$\nu_2(P_k Q_k - k) = \nu_2(2P_k Q_k - 2k) - 1 = \nu_2(P_{2k} - 2k) - 1 \geq 3 - 1 = 2,$$
where the second equality holds since $P_{2k} = 2P_k Q_k$.
\end{proof}

\begin{theorem}\label{thm:curl_b1_main_result}
For all $k \geq 1$, the GCD of all sums of $k$ consecutive cobalancing numbers is
\begin{align*}
    \curl{b}{} = \begin{cases}
  \gcd(P_k, k), & \text{if $k$ is even}; \\
   2\gcd(Q_k, k),  & \text{if $k$ is odd}.
\end{cases} 
\end{align*}
\end{theorem}

\begin{proof}
By Theorem~\ref{thm:curl_b1_intermediary_result}, it suffices to show that $\gcd\left(\frac{1}{2}( B_k - k), P_k \right) = \gcd(P_k, k)$ when $k$ is even, and that $\gcd\left(\frac{1}{2}( B_k - k), 2 Q_k \right) = 2 \gcd(Q_k, k)$ when $k$ is odd.

\begin{itemize}
\item[] \textbf{Case I.} Suppose $k$ is even. We claim that $\nu_p\left(\gcd(\frac{1}{2}(B_k - k),P_k)\right) = \nu_p(\gcd(P_k, k))$ for all primes $p$. Suppose $p\neq 2$ and $j \geq 1$. Then the following sequence of biconditionals holds:
\begin{align*}
    p^j \text{ divides } \gcd\left( \frac{1}{2}(B_k - k), P_k \right)
        &\Longleftrightarrow p^j \text{ divides } P_k \text{ and } B_k - k\\
        &\Longleftrightarrow p^j \text{ divides } P_k \text{ and } k\\
        &\Longleftrightarrow p^j \text{ divides } \gcd(P_k, k),
\end{align*}
where the first biconditional holds since $p$ is an odd prime, and the second one holds since $B_k = P_k Q_k$. Thus $\nu_p\left(\gcd(\frac{1}{2}(B_k - k),P_k)\right) = \nu_p(\gcd(P_k, k))$.

On the other hand, suppose $p=2$. Since $Q_k$ is odd, $k$ is even, and $\nu_2 (P_k) = \nu_2 (k)$ by Lemma~\ref{lem:Jurgen_p_adic_lemma_1}, we have $\nu_2 (P_k Q_k - k) > \nu_2 (P_k)$. It follows that $\nu_2 (P_k Q_k - k) \geq \nu_2 (P_k) + 1$, and so $\nu_2 (P_k Q_k - k) - 1 \geq \nu_2 (P_k)$. Since $B_k = P_k Q_k$, then $\nu_2(B_k - k) = \nu_2(P_k Q_k - k)$, and thus $\nu_2\left(\frac{1}{2}(B_k -k)\right) = \nu_2(P_k Q_k -k) - 1 \geq \nu_2(P_k)$. By Identity~\eqref{eq:p_adic_prop_1} of Lemma~\ref{lem:Carl_is_Filipino}, we have
    \begin{align*}
        \nu_2\left(\gcd\left(\frac{1}{2}(B_k - k), P_k\right)\right) &= \min \left(\nu_2 \left ( \frac{1}{2}(B_k - k) \right ), \nu_2(P_k) \right) \\
        &= \nu_2(P_k) \\
        &= \min \left(\nu_2(P_k), \nu_2(k) \right)\\
        &= \nu_2 (\gcd(P_k, k)).
    \end{align*}

\item[] \textbf{Case II.} Suppose $k$ is odd. We claim that $\nu_p\left(\gcd(\frac{1}{2}(B_k - k), 2 Q_k)\right) = \nu_p(2 \gcd(Q_k, k))$ for all primes $p$. Suppose $p\neq 2$ and $j \geq 1$. Then the following sequence of biconditionals holds:
\begin{align*}
    p^j \text{ divides } \gcd\left( \frac{1}{2}(B_k - k), 2 Q_k \right)
        &\Longleftrightarrow p^j \text{ divides } Q_k \text{ and } B_k - k\\
        &\Longleftrightarrow p^j \text{ divides } Q_k \text{ and } k\\
        &\Longleftrightarrow p^j \text{ divides } \gcd(Q_k, k)\\
        &\Longleftrightarrow p^j \text{ divides } 2\gcd(Q_k, k),
\end{align*}
where the first and fourth biconditionals hold since $p$ is an odd prime, and the second one holds since $B_k = P_k Q_k$. Thus $\nu_p \left(\gcd(\frac{1}{2}( B_k - k), 2Q_k)\right) = \nu_p (2\gcd(Q_k,k))$.

On the other hand, suppose $p=2$. Since $Q_k$ and $k$ are odd, we have $\gcd(Q_k, k)$ is odd and hence $\nu_2 (2\gcd(Q_k, k))=1$. Since $B_k = P_k Q_k$, Lemma~\ref{lem:Jurgen_p_adic_lemma_2} implies that $\nu_2 (B_k - k) \geq 2$, and so $\nu_2 \left(\frac{1}{2}(B_k - k)\right) \geq 1$. By Identity~\eqref{eq:p_adic_prop_1} of Lemma~\ref{lem:Carl_is_Filipino}, we have
    \begin{align*}
        \nu_2 \left(\gcd\left(\frac{1}{2}(B_k - k), 2Q_k\right)\right) = \min \left( \nu_2 \left ( \frac{1}{2} (B_k - k) \right ), \nu_2 (2Q_k)\right) = 1,
    \end{align*}
where the last equality holds since $\nu_2 \left(\frac{1}{2}(B_k - k) \right) \geq 1$, and $Q_k$ being odd implies $\nu_2 (2Q_k)~=~1$. It follows that $\nu_2 \left(\gcd(\frac{1}{2}(B_k-k), 2Q_k)\right) = 1 = \nu_2 (2\gcd(Q_k, k))$.\qedhere
\end{itemize}
\end{proof}

\section{Further directions and open questions}\label{sec:further_results_and_open_questions}

\subsection{GCD of sums of \texorpdfstring{$k$}{k} consecutive squares}
Recall in Convention~\ref{conv:notation_for_curl_GCD_closed_forms}, we use the notation $\curl{S}{m}$ to denote the GCD of all sums of $k$ consecutive $m^{\mathrm{th}}$ powers of sequence terms. Based on evidence collected from the software \texttt{Mathematica}, we have the following closed forms for the GCD of all sums of $k$ consecutive squares of Pell and associated Pell numbers. We present this proof in a future paper.

\begin{result}
For all $k \geq 1$, the GCDs of all sums of $k$ consecutive squares of Pell and associated Pell numbers, respectively, are
$$\curl{P}{2} = \begin{cases}
  \frac{1}{2} P_k,  & \text{if $k$ is even}; \\
   1,  & \text{if $k$ is odd},
\end{cases}
\;\;\;\;\text{and}\;\;\;\;
\curl{Q}{2} = \begin{cases}
  P_k,  & \text{if $k$ is even}; \\
  1,  & \text{if $k$ is odd}.
\end{cases}
$$
\end{result}

\begin{question}
Can we find nice closed forms for $\curl{B}{2}$, $\curl{C}{2}$, $\curl{b}{2}$, and $\curl{c}{2}$? Thus far from evidence collected from \texttt{Mathematica}, we do have the following observations:
\begin{itemize}
\item For all $k \geq 1$, we have $
\curl{B}{2} = \begin{cases}
  \frac{1}{2} \curl{C}{2},  & \text{if $k$ is even}; \\
  \curl{C}{2},  & \text{if $k$ is odd}.
\end{cases}
$
\item For all $k \geq 1$, we have $
\curl{C}{2} = \begin{cases}
  \curl{c}{2},  & \text{if $k \not\equiv 3 \pmod{6}$}; \\
  \frac{1}{3} \curl{c}{2},  & \text{if $k \equiv 3 \pmod{6}$}.
\end{cases}
$
\item For all $k \geq 1$, we have $\curl{b}{2} \equiv 0 \pmod{4}$.
\end{itemize}
\end{question}

\subsection{The GCD of a sequence term and its index}

Recall that by Theorem~\ref{thm:curl_b1_main_result}, we found that
\begin{align*}
    \curl{b}{} = \begin{cases}
  \gcd(P_k, k),  & \text{if $k$ is even;} \\
   2\gcd(Q_k, k),  & \text{if $k$ is odd.}
\end{cases} 
\end{align*}
Unlike the other five GCD closed forms, $\curl{P}{}$, $\curl{Q}{}$, $\curl{B}{}$, $\curl{C}{}$, and $\curl{c}{}$, the formula for $\curl{b}{}$ is a not necessarily a closed form in the same sense, since we still need to compute both $\gcd(P_k, k)$ and $\gcd(Q_k, k)$. This leads to the following question.

\begin{question}
When $\GenericSeq$ is any of the six sequences $\PellSeq$, $\QellSeq$, $\BalSeq$, $\LucBalSeq$, $\CobalSeq$,  or $\LucCobalSeq$, can we find a closed form for the values $\gcd(S_k, k)$ for all $k \geq 1$?
\end{question}

When $\GenericSeq$ is the Pell or associated Pell sequence, we have partial results towards closed forms for $\gcd(S_k, k)$ that involve the \textit{entry point} (or \textit{rank of apparition}), $e_S(p)$, which is the smallest index $r>0$ such that $p$ divides $S_r$ where $p$ is a prime.

\begin{conjecture}
We claim that $\gcd(Q_k, k) > 1$ if and only if there exists a prime $p$ such that $p$ divides $k$ and the rank of apparition (or entry point), $e_Q(p)$ divides $k$. For example, $\gcd(Q_{21}, 21) = 7$ and for the prime $p=7$, we have $p$ divides $21$ and $e_Q(p) = 3$ divides $21$.
\end{conjecture}

\section{Acknowledgments}
The authors acknowledge Florian Luca for his observation to Mbirika regarding Identity~\eqref{eq:Luca_observation} and the known result Identity~\eqref{eq:Luca_known_result} in the Fibonacci setting. Generalizing the latter two identities to the Pell, balancing, and related sequences greatly simplifies many of the proofs in our paper.
\begin{center}
    \includegraphics[width=4.5in]{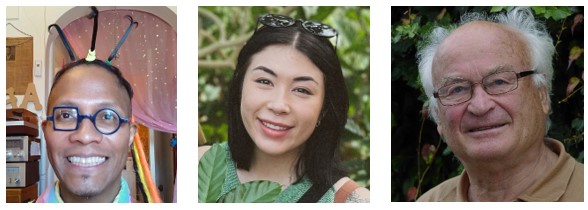}\\
    {\scriptsize aBa Mbirika, Janee Schrader, and J\"urgen Spilker}
\end{center}

\bigskip
\hrule
\bigskip

\noindent 2010 {\it Mathematics Subject Classification}:
Primary 11B39, Secondary 11A05.

\noindent \emph{Keywords: }
Pell sequence, associated Pell sequence, balancing sequence, Lucas-balancing sequence, cobalancing sequence, Lucas-cobalancing sequence, greatest common divisor.

\bigskip
\hrule
\bigskip

\noindent (Concerned with sequences
\seqnum{A000129}, \seqnum{A001333}, \seqnum{A002203},  \seqnum{A053141}, \seqnum{A001541}, \seqnum{A001541}, and \seqnum{A002315}.)

\bigskip
\hrule
\bigskip

\vspace*{+.1in}
\noindent
Received  February 19 2023;
revised versions received  June 13 2023; June 15 2023.
Published in {\it Journal of Integer Sequences}, June 17 2023.

\bigskip
\hrule
\bigskip

\noindent
Return to \href{https://cs.uwaterloo.ca/journals/JIS/}{Journal of Integer Sequences home page}.
\vskip .1in

\end{document}